\documentclass[12pt]{article}
\usepackage{amsmath}
\usepackage{amsfonts}
\usepackage{amssymb}
\usepackage{amsthm}
\usepackage{amstext}
\usepackage{color}
\usepackage{graphicx}
\usepackage{hyperref}

\usepackage{color}

\newtheorem{theorem}{Theorem}

\newtheorem{corollary}[theorem]{Corollary}

\newtheorem{lemma}[theorem]{Lemma}

\newtheorem{proposition}[theorem]{Proposition}
\newtheorem{remark}[theorem]{Remark}

\def\div{\mbox{div}\,}

\def\R{\mathbb{R}}

\def\Sc{\mbox{Sc}\,}
\def\Vec{\mbox{Vec}\,}

\def\tr{\operatorname{\rm tr}}

\begin{document}

\title{Calderón's inverse problem via Vekua theory}
\author{Briceyda B. Delgado\thanks{INFOTEC, Centro de Investigación e Innovación en Tecnologías de la Información y Comunicación
\\Cto. Tecnopolo Sur No. 112, Pocitos, Aguascalientes 20326, Mexico}}
\date{}
\maketitle
\begin{abstract}
In this work, we will prove a uniqueness result for Calderón's inverse problem via some integral representation formulas for solutions of the Vekua equation in the framework of Clifford analysis.
\end{abstract}
\mbox{}
\mbox{}\\
\noindent \textbf{Classification:} 35R30, 30G20. 


\section{Introduction}
Let us consider the conductivity equation in a bounded domain $\Omega\subset \mathbb{R}^n$, $n\geq 3$, with Lipschitz boundary
\begin{equation}\label{eq:conductivity}
  \div \left(\sigma \nabla u_0\right)=0, \qquad \text{ in }\Omega, 
\end{equation}
where $\sigma$ is a non-vanishing conductivity function in $W^{1,\infty}(\Omega)$. Given $\varphi_0\in H^{1/2}(\partial\Omega)$, then \eqref{eq:conductivity} has a unique weak solution $u_0$ in $H^1(\Omega)$ such that $u_0|_{\partial\Omega}=\varphi_0$.

The Calderón's inverse conductivity problem is about recovering the conductivity $f$ from the Dirichlet-to-Neumann map defined as
\begin{align}\label{eq:D-N}
\nonumber \Lambda_{\sigma}\colon H^{1/2}(\partial \Omega)&\rightarrow H^{-1/2}(\partial \Omega),\\
  \varphi_0 &\mapsto \sigma \nabla u_0|_{\partial \Omega}\cdot \eta,
\end{align}
where $\eta$ is the outward unit normal vector to $\partial\Omega$ and $u_0\in H^1(\Omega)$ is solution of the conductivity equation \eqref{eq:conductivity} with $u_0|_{\partial\Omega}=\varphi_0$.
There are several aspects that are interesting both from mathematical theory and practical applications. In this manuscript, we will address the uniqueness result in Section \ref{sec:uniqueness}, which, in short, consists of showing that the map $\sigma\mapsto \Lambda_{\sigma}$ is injective. In 1980, Calderón proved that the inverse conductivity problem, also known as the electrical impedance tomography problem, admits a clear and precise mathematical formulation through the conductivity equation \cite{Calderon}.

In the complex case, Astala and P\"aiv\"{a}rinta proved the uniqueness for $L^{\infty}(\Omega)$ conductivities away from zero and infinity \cite{AstPai2006} and the reconstruction in \cite{Astala-2}.

In higher dimensions, $n\geq 3$, it has been more than four decades of interesting works concerning the uniqueness result for Calder\'{o}n problem, for instance, Kohn and Vogelius proved uniqueness for real analytic conductivities \cite{Vogelius}, Sylvester and Uhlmann proved for smooth conductivities \cite{Sylvester-Uhlmann-1986, Sylvester-Uhlmann, Sylvester-Uhlmann-1988}, Nachman, Sylvester and Uhlmann improved for twice continuously differentiable conductivities \cite{Nachman}. 
This was followed by Brown \cite{Brown-1996} 
who relaxed the regularity of the conductivity to $3/2+\epsilon$ derivatives. Later, P\"aiv\"arinta, Panchenko, and Uhlmann proved the uniqueness result for conductivities in $W^{3/2,\infty}(\Omega)$ \cite{Paivarinta2003}, see also the work of Brown and Torres \cite{Russell-1}. More recently, Haberman and Tataru proved uniqueness for Lipschitz conductivities with logarithmic gradient sufficiently small \cite{Haberman-Tataru}, and in 2016, Caro and Rogers removed the smallness constraint from the uniqueness result in \cite{Caro-Rogers}.

In this work, we aim not only to strengthen the understanding of the uniqueness result for the Calderón problem but also to develop tools that could be adapted to related inverse problems, for instance, for the anisotropic Maxwell equation \cite{Salo-2009}.
In particular, we continue the work developed in \cite{Delgado2023}, where we analyze the space of solutions of the \textit{Vekua equation} 
\begin{align}\label{eq:Vekua_equation}
   D w=\alpha\, \overline{w}, \quad \text{ in }\Omega,
\end{align}
in the framework of Clifford analysis. It was proved that for the \textit{main Vekua equation}. That is, in the particular case when $\alpha=\nabla f/f$, where $f$ a non-vanishing function in $W^{1,\infty}(\Omega)$, there exists an interrelation between the solutions of \eqref{eq:Vekua_equation} and the  solutions of the \textit{Clifford Beltrami equation}
\begin{equation}\label{eq:Clifford-Beltrami}
    Du=\dfrac{1-f^2}{1+f^2} D\overline{u}\qquad \text{in }\Omega,
\end{equation}
where $u=f^{-1}\sum_{k\equiv 0,3 (\text{mod } 4)} [w]_k+f\, \sum_{k\equiv 1,2 (\text{mod } 4)} [w]_k$, see Proposition \ref{prop:Beltrami equivalence}. Observe that the Beltrami coefficient $\mu:=(1-f^2)/(1+f^2)$ coincides with the one from Astala and P\"aiv\"arinta \cite{AstPai2006, Astala-2} in the complex case. 
Spaces of solutions of the Vekua equation had been extensively studied in the complex case from the perspective of Hardy classes in \cite{BLRR2010, BFL2011}, and from Bergman classes in \cite{DelLeb2019, Vicente2025}. In \cite{Santa2019}, Santacesaria noted that a natural framework for analyzing this inverse problem is via Clifford algebras; the main result of that work showed the equivalence between solutions of the Clifford Beltrami equation \eqref{eq:Clifford-Beltrami} and solutions of a div-curl system, which is in turn equivalent to analyzing the solutions of the main Vekua equation in the quaternionic case. Simultaneously, some quaternionic analysis techniques were used to construct solutions and to provide a Hilbert transform for the main Vekua equation in the three-dimensional case, see \cite{DelPor2018, DelPor2019}.

The outline of this paper is as follows. In Section \ref{sec:Preliminaries}, we introduce some necessary definitions of Clifford analysis and recall some results obtained in \cite{Delgado2023} for spaces of solutions of the Vekua equation.
In Section \ref{sec:integral}, we establish a generalized Borel-Pompeiu integral formula involving the Vekua operators $D-\alpha C$ and $D-M^{\alpha}C$, see Theorem \ref{th:Cauchy-theorem}. Moreover, a partial Cauchy integral formula is obtained
\begin{align}\label{eq:scalar-Cauchy}
  \text{Sc }\int_{\partial\Omega} \Phi(y-x)\, \eta(y) \, w(y) \,ds_{y}=\left\{
   \begin{array}{ll}
      w_0(x), \ & x\in \Omega,\\
      0,           & x\in \R^n\setminus \overline{\Omega},
   \end{array}
\right.
\end{align}
where $\Phi$ is a fundamental solution of $D-M^{\alpha}C$, $w$ is a solution of the Vekua equation \eqref{eq:Vekua_equation}, and $w_0$ the scalar part of $w$, see Corollary \ref{cor:Cauchy}. Moreover, Subsection \ref{subsec:main-Vekua} is devoted to the main Vekua equation and we enunciate some different integral formulas where naturally appear the associated Dirichlet-to-Neumann maps associated to the conductivity and to the Schr\"{o}dinger equation, respectively.


In Section \ref{sec:uniqueness}, we revisit the uniqueness problem for the Calderón problem for conductivities in $W^{2,\infty}(\Omega)$ and $W^{1,\infty}(\Omega)$, respectively; and we propose new perspectives to its analysis. Our methods offer alternative proofs of well-known results while providing a framework that may extend to broader classes of conductivities. 

In Proposition \ref{prop:difference-potentials} we generalize the classical identity in terms of fundamental solutions of the Schr\"{o}dinger operator
\begin{align*}
\int_{\Omega} \left(\frac{\Delta g}{g}-\frac{\Delta f}{f}\right) w_{0,f}(y)\theta_g(y-x)\, dy&=\int_{\Omega} \left(\frac{\Delta g}{g}-\frac{\Delta f}{f}\right) w_{0,g}(y)\theta_f(y-x)\, dy \\
&=\left\{
   \begin{array}{ll}
      w_{0,f}(x)-w_{0,g}(x), \ & x\in \Omega,\\
      0,           & x\in \R^n\setminus \overline{\Omega},
   \end{array}
\right.
\end{align*}
where $w_{0,f}, w_{0,g}$ are solutions and $\theta_f, \theta_g$ are fundamental solutions of the Schr\"{o}dinger equation with potential $q_f=\Delta f/f$ and $q_g=\Delta g/g$, respectively.

Besides, we have proved in Proposition \ref{prop:extensions-equal} that the unique extensions in $H^1(\Omega)$ as solutions of the conductivity equation are uniquely determined by the boundary measurements provided by the Dirichlet-to-Neumann map, complementing the identifiability result given in \cite{Sylvester-Uhlmann} that establishes that the complex geometrical solutions are equal in the exterior domain $\mathbb{R}^n\setminus \overline{\Omega}$. 

The aim of this paper is to give an alternative proof of the uniqueness result for the Calderón problem with Lipschitz conductivities. More precisely, we prove 
\begin{theorem}[\textbf{Uniqueness result}]\label{th:uniqueness}
Let $n\geq 3$. Let $f, g\in W^{1,\infty}(\Omega)$ be conductivities away from zero in $\Omega$. If $\Lambda_f=\Lambda_g$, then $f=g$ in $\Omega$.
\end{theorem}


\section{Preliminaries}\label{sec:Preliminaries}
Let us consider the real Clifford algebra $\mathcal{C}l_{0,n}$ generated by the elements $e_0=1,e_1,e_2,\cdots,e_n$ with the relation $e_ie_j+e_je_i=-2\delta_{ij}$, where 
$\delta_{ij}$ is the Kronecker delta function $(i,j=1,2,\cdots,n)$. We define the \textit{conjugation} in $\mathcal{C}l_{0,n}$ as usual $\overline{ab}=\overline{b}\overline{a}$, $\forall a,b\in \mathcal{C}l_{0,n}$.
Let us denote by $a=\sum_{A} a_A e_A\in \mathcal{C}l_{0,n}$ the elements in the real Clifford algebra, and let us define the following projections for $0\leq k\leq n$ as follows
\begin{equation*}
[a]_k=\sum_{|A|=k} a_A\, e_A.
\end{equation*}
Thus, an arbitrary element $a\in \mathcal{C}l_{0,n}$ can be written as
\begin{equation*}
a=[a]_0+[a]_1+\cdots +[a]_n.
\end{equation*}
We will define the scalar, vector, paravector, and the non-paravector part of $a$ as follows
\begin{align*}
\text{Sc }a=[a]_0,\quad \text{Vec }a=[a]_1, \quad \text{Pa }a=[a]_0+[a]_1, \quad \text{ and } \quad \text{NPa }a=[a]_2+\cdots+[a]_n, 
\end{align*}
respectively. For short, we use $a_0=\text{Sc }a$ and $\vec a=\text{Vec }a$. We embed the Euclidean space $\R^{n+1}$ in $\mathcal{C}l_{0,n}$ through the identification of $x^*=(x_0,x_1,\cdots,x_n)\in \R^{n+1}$ with the Clifford \textit{paravector} $x^*=x_0+\sum_{i=1}^n x_ie_i$. 
Through all the manuscript $\Omega\subset \R^n$ will be a bounded domain with Lipschitz boundary, and the elements $x\in \Omega$ will be called \textit{vectors} with $x=\sum_{i=1}^n x_ie_i$.

 In this work, we will work with $\mathcal{C}l_{0,n}$-valued functions defined in $\Omega$, $w \colon \Omega \to \mathcal{C}l_{0,n}$ denoted by
\[ 
w(x)=\sum_A w_A(x)\, e_A, \quad x\in \Omega,
\]
where the coordinates $w_A$ are real-valued functions defined in $\Omega$, that is $w_A\colon \Omega \to \R$. 
In particular, paravector-valued functions are denoted by $w(x)=\sum_{i=0}^n w_i(x) e_i$. For details on Clifford analysis, see the monographs \cite{BDS1982, GuSpr1997, GuHaSpr2008}.

Let us denote by $L^p(\Omega,C\ell_{0,n})$ the set of $C\ell_{0,n}$-valued functions defined in $\Omega$ which are $p$-integrable $1<p<\infty$. In particular, $L^2(\Omega,C\ell_{0,n})$ is a right Hilbert $C\ell_{0,n}$-module under the inner product
\begin{align}\label{eq:inner-product}
  \langle u,v\rangle_{L^2(\Omega)}=\int_{\Omega}\overline{u} v \, dy, \quad \forall u,v\in L^2(\Omega,C\ell_{0,n}).
\end{align}
Since we consider spaces of solutions which are not subspaces over $C\ell_{0,n}$, the following inner product is considered
\begin{equation}\label{eq:scalar-product}
   \langle u,v\rangle_{0,L^2}=\text{Sc }\int_{\Omega} \overline{u}v\, dy. 
\end{equation}
The \textit{Moisil-Teodorescu} differential operator, defined by 
\begin{align}\label{eq:Moisil-Teodorescu}
  D=\sum_{i=1}^n e_i \, \partial_i,
\end{align}
where $\partial_i$ means the partial derivative with respect to the variable $x_i$. We will say that $w\in C^1(\Omega,C\ell_{0,n})$ is \textit{monogenic} if $D w=0$ in $\Omega$. By the relation $\Delta=-D^2$, it is straightforward that any monogenic function is harmonic component-wise. 
We will use the following properties throughout this work
\[
\text{Sc }Dw=-\div \vec w, \qquad \div \Vec Dw=0, \qquad D(\nabla w_0)=-\Delta w_0,
\]
where $\text{div}$ is the classical divergence operator in $\mathbb{R}^n$.
For $n\geq 3$. Let us define the Cauchy kernel $E$ as follows
\begin{equation}\label{eq:Cauchy-kernel}
E(x)=-\frac{x}{\sigma_n|x|^n}, \quad x\in \R^n\setminus\{0\},
\end{equation}
and $\sigma_n$ is the surface area of the unit sphere in $\R^n$.

We define the \textit{Teodorescu transform} \cite{GuHaSpr2008}, for $w\in L^p(\Omega,C\ell_{0,n})$ as follows
\begin{align}\label{eq:Teodorescu_operator}
   T_{\Omega}[w](x)=-\int_{\Omega} 
	  E(y-x) w(y)\, dy, \quad x\in\R^n,
\end{align}
where the Cauchy kernel $E$ was previously defined in \eqref{eq:Cauchy-kernel}. Moreover, $T_{\Omega}\colon L^p(\Omega, C\ell_{0,n})\to L^p(\Omega, C\ell_{0,n})$ is a bounded operator and it is a right inverse of the Moisil-Teodorescu operator $D$. 

\subsection{Generalized Vekua spaces}\label{sec:Bergman-Vekua spaces}
Let us denote by $A^p(\Omega,C\ell_{0,n})$ the subspace of monogenic functions in $L^p(\Omega,C\ell_{0,n})$, $1<p<\infty$, which is also a right $C\ell_{0,n}$-module. Due to $A^p(\Omega,C\ell_{0,n})$ is a closed subspace in $L^p(\Omega, C\ell_{0,n})$ $1<p<\infty$ \cite[Prop.\ 3.73]{GuSpr1997}.
Therefore, the right Hilbert $C\ell_{0,n}$-module $L^2(\Omega, C\ell_{0,n})$ allows an orthogonal decomposition
\[
   L^2(\Omega, C\ell_{0,n})=A^2(\Omega,C\ell_{0,n})\oplus DW_0^{1,2}(\Omega,C\ell_{0,n}),
\]
under the scalar product \eqref{eq:inner-product}.




Let $1<p<\infty$ and let $\alpha\in L^{\infty}(\Omega,C\ell_{0,n})$ be a bounded function on $\Omega$. Following \cite{Delgado2023}, we define the \textit{generalized Vekua space} $A_{\alpha}^p(\Omega,C\ell_{0,n})$ to consist of $C\ell_{0,n}$-valued functions $w \in L^p(\Omega,C\ell_{0,n})$ satisfying the Vekua equation \eqref{eq:Vekua_equation} in the sense of distributions. Some monographs concerning Vekua theory are \cite{Vekua1962, Bers1953, Krav2009}

Let $M^{\alpha}$ be the right-hand side multiplication operator, $M^{\alpha}[w]:=w\, \alpha$. In \cite[Th.\ 11]{Delgado2023} it was proved that the Hilbert space $L^2(\Omega,C\ell_{0,n})$ allows the orthogonal decomposition 
\begin{align}\label{eq:Hodge-2}  
L^2(\Omega,C\ell_{0,n})=A_{\alpha}^2(\Omega,C\ell_{0,n})\oplus (D-M^{\alpha}C) W_0^{1,2}(\Omega,C\ell_{0,n}),
\end{align} 
where $C$ represents the Clifford conjugation operator, and the orthogonality is considered with respect to the scalar product \eqref{eq:scalar-product}.

In the literature, \eqref{eq:Hodge-2} is also called a \textit{Hodge decomposition} for generalized Vekua spaces.
One key point for the analysis of generalized Vekua spaces was the good properties of the operator $S_{\alpha}$ defined in terms of the Teodorescu transform as follows
\begin{align}\label{eq:invariant-operator}
S_{\alpha}\colon A_{\alpha}^2(\Omega,C\ell_{0,n})&\rightarrow A^2(\Omega,C\ell_{0,n})\\
\nonumber S_{\alpha}&=I-T_{\Omega}[\alpha C], 
\end{align}
which sends solutions of the Vekua equation \eqref{eq:Vekua_equation} into monogenic functions in $\Omega$. In the following, we will analyze some integral representation formulas associated with these Vekua operators $D-\alpha C$ and $D-M^{\alpha}C$.


\section{Integral representation formulas}\label{sec:integral}
The following integral formulas \eqref{eq:scalar-BP}, \eqref{eq:scalar-BP-2} are partial generalizations of the Borel-Pompeiu formula in Clifford analysis; see, for instance, \cite[Th.\ 7.8]{GuHaSpr2008}, which establishes that any $v\in C^1(\overline{\Omega})$ satisfies
\begin{align}\label{eq:formula_BP}
   T_{\Omega}[Dv](x)+\int_{\partial\Omega}E(y-x)\eta(y)v(y)\, ds_y= \left\{
   \begin{array}{ll}
      v(x), \ & x\in \Omega,\\
      0,           & x\in \R^n\setminus \overline{\Omega},
   \end{array}
\right.
\end{align} 
where the Teodorescu operator was previously defined in \eqref{eq:Teodorescu_operator}. In contrast to \eqref{eq:formula_BP}, we would be able to recover only the scalar part of the solutions of the Vekua equation \eqref{eq:Vekua_equation}.

Let $\overline{\Phi}$ and $\overline{\Psi}$ be normalized fundamental solutions of $D-M^{\alpha}C$ and $D-\alpha C$, respectively. Since we do not know explicit forms of fundamental solutions for the Vekua operator $D-\alpha C$ nor for the operator $D-M^{\alpha}C$, we will work with the following standard constraint. Let $\epsilon>0$ and let us consider $B(x;\epsilon)=\{y\in \Omega \colon  |x-y|<\epsilon\}$ such that $\overline{B(x;\epsilon)}\subset \Omega$, then for all $v\in C^1(\overline{\Omega})$
\begin{align}\label{constraint-1}
   \lim_{\epsilon\rightarrow 0} \int_{\partial B(x;\epsilon)} \Phi(y-x)\eta(y)v(y)\, ds_y=v(x),\\
   \lim_{\epsilon\rightarrow 0} \int_{\partial B(x;\epsilon)} \Psi(y-x)\eta(y)v(y)\, ds_y=v(x).\label{constraint-2}
\end{align}
\begin{theorem}\label{th:Cauchy-theorem}
 Let $\Omega\subset \mathbb{R}^n$ be a bounded domain with Lipschitz boundary. Suppose that $ u, v\in C^1(\overline{\Omega})$, and $\alpha\in L^{\infty}(\Omega)$. Then the following formulas are fulfilled
\begin{align}
   \text{Sc }\int_{\partial\Omega} u\, ds_y^* \, v&=\text{Sc }\int_{\Omega} u [ (D-\alpha C)v]-
   [(D-M^{\alpha}C)\overline{u}]\, \overline{v} \, dy\label{eq:integral-formula-1}\\
   \text{Sc }\int_{\partial\Omega} u\, ds_y^* \, v&=\text{Sc }\int_{\Omega} u [(D-M^{\alpha}C)v]-
   [(D-\alpha C)\overline{u}]\, \overline{v}\, dy.\label{eq:integral-formula-2}
\end{align}
Moreover, if $-\overline{\Phi}$ is a fundamental solution of the operator $D-M^{\alpha}C$ satisfying \eqref{constraint-1}, then we can recover the scalar part of $v$, $v_0$, as follows
\begin{align}\label{eq:scalar-BP}
\Sc\, \int_{\partial\Omega}\Phi(y-x) \eta(y) v(y)\, ds_{y}-\Sc\, \int_{\Omega} \Phi(y-x) [ (D-\alpha C)v] \, dy=\left\{
   \begin{array}{ll}
      v_0(x), \ & x\in \Omega,\\
      0,           & x\in \R^n\setminus \overline{\Omega}.
   \end{array}
\right. 
\end{align}
Analogously,  if $-\overline{\Psi}$ is a fundamental solution of the operator $D-\alpha C$ satisfying \eqref{constraint-2}, then 
\begin{align}\label{eq:scalar-BP-2}
  \Sc\, \int_{\partial\Omega}\Psi(y-x) \eta(y) v(y)\, ds_{y}-\Sc\, \int_{\Omega} \Psi(y-x) [ (D-M^{\alpha} C)v] \, dy=\left\{
   \begin{array}{ll}
      v_0(x), \ & x\in \Omega,\\
      0,           & x\in \R^n\setminus \overline{\Omega}.
   \end{array}
\right. 
\end{align}
\end{theorem}
\begin{proof}
Let $x\in \Omega$ be a fixed point. Let $\Omega^*\subseteq \Omega$. Applying Gauss's theorem to $u, v$ \cite[Th.\ A.2.22]{GuHaSpr2008} we get
\begin{align*}
\int_{\partial\Omega^*} u\,\eta \, v\, ds_y&=\int_{\Omega^*} \left((uD)v+u(Dv)\right)\, dy =\int_{\Omega^*} \left((uD)v+u\alpha\overline{v}+u(Dv-\alpha \overline{v})\right)\, dy.
\end{align*}
Using that $\Sc((uD)v)=\Sc(-D\overline{u}\,\overline{v})$ and taking the scalar part of the previous expression, we get that
\begin{align*}
Sc \, \int_{\partial\Omega^*} u(y)\eta(y) v(y)\, ds_y&=\Sc\, \int_{\Omega^*} [(-D+M^{\alpha}C)\overline{u}]\, \overline{v}+u(D-\alpha C)v \, dy,
\end{align*}
which coincides with \eqref{eq:integral-formula-1}. Analogously, we can obtain \eqref{eq:integral-formula-2}.
Choose $\epsilon>0$ such that $\overline{B(x;\epsilon)}\subset \Omega$, where $B(x;\epsilon)=\{y\in \Omega \colon  |x-y|<\epsilon\}$. Applying \eqref{eq:integral-formula-1} with $u(y)=\Phi(y-x)$ in the domain $\Omega_{\epsilon}:=\Omega\setminus \overline{B(x;\epsilon)}$, we get
\begin{align*}
\Sc \, \int_{\partial\Omega_{\epsilon}}\Phi(y-x) \eta(y) v(y)\, ds_{y}=\Sc \, \int_{\Omega_{\epsilon}} \Phi(y-x) [ (D-\alpha C)v]-
   [(D-M^{\alpha}C)\overline{\Phi}(y-x)]\, \overline{v} \, dy.
\end{align*}
Since $D\overline{\Phi}(y-x)=\Phi(y-x) \alpha$ in $\Omega_{\epsilon}$, then the previous expression reduces to
\begin{align}\label{eq:reduction}
\Sc \int_{\partial\Omega_{\epsilon}}\Phi(y-x) \eta(y) v(y)\, ds_{y}=\Sc \int_{\Omega_{\epsilon}} \Phi(y-x) [ (D-\alpha C)v] \, dy.
\end{align}
Taking the limit when $\epsilon\rightarrow 0$ and using \eqref{constraint-1}, we get
\begin{align*}
\lim_{\epsilon \rightarrow 0}\int_{\partial\Omega_{\epsilon}}\Phi(y-x) \eta(y) v(y)\, ds_y=\int_{\partial\Omega}\Phi(y-x) \eta(y) v(y)\, ds_y-v(x).
\end{align*}
Then, taking the limit on both sides of \eqref{eq:reduction}, we readily obtain \eqref{eq:scalar-BP}. Similarly, applying \eqref{eq:integral-formula-2} with $u(y)=\Psi(y-x)$ in the domain $\Omega_{\epsilon}:=\Omega\setminus \overline{B(x;\epsilon)}$ we get \eqref{eq:scalar-BP-2} as we desired.
\end{proof}
The following result shows that we can recover the scalar part of any solution of the Vekua equation \eqref{eq:Vekua_equation} when we know the boundary value of the solution, so we can see this result as a partial generalization of the classical Cauchy integral formula for monogenic functions.   
\begin{corollary}\label{cor:Cauchy}
Under the same hypothesis as Theorem \ref{th:Cauchy-theorem}. If $-\overline{\Phi}$ is a fundamental solution of the operator $D-M^{\alpha}C$ satisfying \eqref{constraint-1}, and $w\in\text{Ker }(D-\alpha C)$ whose boundary values are well-defined, then we have a Cauchy-type integral formula in $\Omega$
\begin{align}\label{eq:scalar-Vekua}
  \text{Sc }\int_{\partial\Omega} \Phi(y-x)\, \eta(y) \, w(y) \,ds_{y}=\left\{
   \begin{array}{ll}
      w_0(x), \ & x\in \Omega,\\
      0,           & x\in \R^n\setminus \overline{\Omega},
   \end{array}
\right.
\end{align}
where $w_0(x)=\text{Sc }w(x)$. Analogously,  if $-\overline{\Psi}$ is a fundamental solution of the operator $D-\alpha C$, satisfying \eqref{constraint-2}, and $w^*\in \text{Ker }(D-M^{\alpha}C)$ whose boundary values are well-defined, then we have a Cauchy-type integral formula in $\Omega$
\begin{align}\label{eq:scalar-adjoint}
  \text{Sc }\int_{\partial\Omega} \Psi(y-x)\, \eta(y) \, w^*(y) \,ds_{y}=\left\{
   \begin{array}{ll}
      w_0^*(x), \ & x\in \Omega,\\
      0,           & x\in \R^n\setminus \overline{\Omega},
   \end{array}
\right.
\end{align}
where $w_0^*(x)=\text{Sc }w^{*}(x)$.
\end{corollary}

\begin{corollary}\label{cor:Cauchy-fundamental}
Under the same hypothesis than Theorem \ref{th:Cauchy-theorem}. If $h_0$ is a scalar fundamental solution of $-\Delta+|\alpha|^2+\div \vec \alpha$, with $D\alpha=-\div \vec \alpha$, then
\begin{align*}
  \text{Sc }\int_{\partial\Omega} (\nabla+\overline{\alpha})h_0(y-x)\, \eta(y) \, w(y) \,ds_{y}=\left\{
   \begin{array}{ll}
      w_0(x), \ & x\in \Omega,\\
      0,           & x\in \R^n\setminus \overline{\Omega},
   \end{array}
\right.
\end{align*}
for any $w$ solution of the Vekua equation \eqref{eq:Vekua_equation} whose boundary values are well-defined.
\end{corollary}
\begin{proof}
First, we will verify the following factorization
\begin{align}\label{eq:factorization-2}
    \left(-\Delta+|\alpha|^2-D\alpha\right)h_0=(D-M^{\alpha}C)(D-\alpha C)h_0.
\end{align}
Indeed,
\begin{align*}
(D-M^{\alpha}C)(D-\alpha C)h_0&=(D-M^{\alpha}C)(Dh_0-\alpha h_0)=D^2 h_0-D(\alpha h_0)-\overline{Dh_0}\alpha+\overline{\alpha h_0}\alpha\\
&=-\Delta h_0-D(\alpha) h_0- \nabla h_0\alpha+\nabla h_0\alpha+|\alpha|^2h_0.
\end{align*}

By \eqref{eq:factorization-2}, if $h_0$ is a fundamental solution of $-\Delta+|\alpha|^2-D \alpha=-\Delta+|\alpha|^2+\div \vec \alpha$, then $\overline{\Phi}=(D-\alpha C)h_0$ will be a fundamental solution of $D-M^{\alpha}C$.
The result comes directly from \eqref{eq:factorization-2} and the Cauchy-type integral formula provided in \eqref{eq:scalar-Vekua}.
\end{proof}

\subsection{Integral formulas for the main Vekua equation}\label{subsec:main-Vekua}
From now on, we will consider $\alpha=\nabla f/f$,  with $f$ a scalar function in $W^{1,\infty}(\Omega)$ away from zero (i.e. there exists $C>0$ such that $|f(x)|\geq C>0$ $\forall x\in \Omega$). More precisely, we will analyze the solutions of the \textit{main Vekua equation}
\begin{align}\label{eq:main-Vekua}
 Dw=\frac{\nabla f}{f}\overline{w}, \qquad \text{ in }\Omega.
\end{align}
After seeing the interrelation between the Vekua operator $D-\alpha C$ and the operator $D-M^{\alpha}C$ illustrated in the Hodge decomposition \eqref{eq:Hodge-2} and the previous integral representation formulas (see Theorem \ref{th:Cauchy-theorem}), we would like to analyze the solutions of
\begin{equation}\label{eq:Vekua-adjoint}
  Dw=\overline{w}\frac{\nabla f}{f}, \qquad \text{ in }\Omega.    
\end{equation}

Observe that $D\alpha=D(\nabla f/f)=|\nabla f|^2/f^2-\Delta f/f$. Consequently, the left-hand side of \eqref{eq:factorization-2} reduces to the Schrödinger equation
\begin{equation}\label{eq:factorization-3}
    \left(-\Delta+\dfrac{\Delta f}{f}\right)h_0=\left(D-M^{\frac{\nabla f}{f}}C\right)\left(D-\dfrac{\nabla f}{f} C\right)h_0,
\end{equation}
which in turn is the second-order partial differential equation satisfied by the scalar part of solutions of the main Vekua equation (see \cite[Cor.\ 1]{Delgado2023}). Similar connections between Vekua type operators and Schr\"{o}dinger operator have been considered, for instance, in \cite{Bernstein1996, Krav2006}. Now, we would be able to reconstruct the scalar part of the solutions of the main Vekua equation \eqref{eq:main-Vekua} when we know the boundary values of its solutions.
\begin{corollary}\label{cor:main-Cauchy}
Let $\Omega\subset \mathbb{R}^n$ be a bounded domain with Lipschitz boundary. Let $f\in W^{1,\infty}(\Omega)$ away from zero. If $h_0$ is a fundamental solution of the Schrödinger operator $-\Delta+\Delta f/f$, then
\begin{align*}
\text{Sc }\int_{\partial\Omega} \left(D-\dfrac{\nabla f}{f}\right)h_0(y-x)\, \eta(y) \, w(y) \,ds_{y}=\left\{
   \begin{array}{ll}
      w_0(x), \ & x\in \Omega,\\
      0,           & x\in \R^n\setminus \overline{\Omega},
   \end{array}
\right.
\end{align*}
where $w$ is a solution of the main Vekua equation \eqref{eq:main-Vekua} whose boundary values are well-defined.
\end{corollary}
We will denote the $C\ell_{0,n}$-valued solutions of \eqref{eq:main-Vekua} and \eqref{eq:Vekua-adjoint} as follows
\[
w=[w]_0+[w]_1+[w]_2+\cdots+[w]_n,
\]
where $[w]_k=\sum_{|A|=k} w_A\, e_A$. Equivalently, we can rewrite it as 
\begin{equation}\label{eq:equivalent-expression}
w=\sum_{k\equiv 0,3 (\text{mod } 4)} [w]_k+\sum_{k\equiv 1,2 (\text{mod } 4)} [w]_k, \quad 0\leq k\leq n.
\end{equation}
Notice that if $k\equiv 0,2$ (mod 4), then $k$ is even and if $k\equiv 1,3$ $(\text{mod }4)$ then $k$ is odd. This representation \eqref{eq:equivalent-expression} is useful because the action of the conjugate operator is straightforward, indeed
\begin{align}\label{eq:conjugate}
\overline{w}=\sum_{k\equiv 0,3 (\text{mod } 4)} [w]_k\,-\sum_{k\equiv 1,2 (\text{mod } 4)} [w]_k, \quad 0\leq k\leq n.
\end{align}
The following result illustrates the equivalence between the solutions of the main Vekua equation \eqref{eq:main-Vekua} with the \textit{Clifford Beltrami equation} \eqref{eq:Clifford-Beltrami}.
\begin{proposition}\cite[Prop.\ 2]{Delgado2023}\label{prop:Beltrami equivalence}
Let $\alpha=\nabla f/f$, with $f$ be non-vanishing scalar function in $ W^{1,\infty}(\Omega)$. Then $w=\sum_{k\equiv 0,3 (\text{mod } 4)} [w]_k+\sum_{k\equiv 1,2 (\text{mod } 4)} [w]_k$ is solution of \eqref{eq:main-Vekua} in $\Omega$ if and only if
\begin{equation}\label{eq:transformation-solution}
  u=\dfrac{1}{f}\sum_{k\equiv 0,3 (\text{mod } 4)} [w]_k+f\, \sum_{k\equiv 1,2 (\text{mod } 4)} [w]_k,
\end{equation}
is solution of the Clifford Beltrami equation \eqref{eq:Clifford-Beltrami} in $\Omega$.
\end{proposition}
The proof relies heavily on the next equivalence: $w=\sum_{k\equiv 0,3 (\text{mod } 4)} [w]_k+\sum_{k\equiv 1,2 (\text{mod } 4)} [w]_k$ is solution of \eqref{eq:main-Vekua} in $\Omega$ if and only if 
\begin{align}\label{eq:equivalence-Vekua}
 f\, D\left(\dfrac{1}{f} \sum_{k\equiv 0,3 (\text{mod } 4)} [w]_k \right)+\dfrac{1}{f} D\left(f\, \sum_{k\equiv 1,2 (\text{mod } 4)} [w]_k \right)=0,\qquad \text{in }\Omega.
\end{align}



Let $w$ be a solution of the main Vekua equation \eqref{eq:main-Vekua}. Observe that taking the vector part of the equivalent system \eqref{eq:equivalence-Vekua}, we obtain 
\[
 \Vec D\left(f [w]_2\right)=-f^2\nabla(w_0/f).
\]
Taking the divergence in both sides, we get that $w_0/f$ satisfies the \textit{conductivity equation}
\begin{equation}\label{eq:conductivity-Vekua}
    \div(f^2 \nabla (w_0/f))=0, \qquad \text{ in }\Omega.
\end{equation}
In particular, if the bivector part of $w$ vanishes, that is if $[w]_2= 0$, then $w_0/f$ is constant, and the conductivity equation trivially holds in $\Omega$. We have thus proven the first part of the following result:
\begin{proposition}(\textbf{Green-Vekua formula})\label{prop:Green-Vekua} \cite{Delgado-Moreira}
Let $\Omega\subset \mathbb{R}^n$ be a bounded domain with Lipschitz boundary and $f\in W^{1,\infty}(\Omega)$ away from zero. If $w$ is a solution of the main Vekua equation of \eqref{eq:main-Vekua}, then $\text{Sc w}/f=w_0/f$ satisfies the conductivity equation \eqref{eq:conductivity-Vekua} in $\Omega$. Moreover, if $w_0/f$ has boundary values and normal derivative well-defined, then the following Green-Vekua formula holds
\begin{equation}\label{eq:Green-Vekua-2}
\int_{\partial \Omega}  \left(- \vec \Phi(y-x)\cdot \eta(y)\, w_0(y)+\phi_0(y-x)\, f^2 \nabla\left(\dfrac{w_0}{f}\right)\Big|_{\partial\Omega}\cdot \eta\right) ds_{y}=\left\{
   \begin{array}{ll}
      w_{0}(x), \ & x\in \Omega,\\
      0,           & x\in \R^n\setminus \overline{\Omega},
   \end{array}
\right.
\end{equation}
where $\vec \Phi$ is a purely vector fundamental solution of $D-M^{\nabla f/f} C$ satisfying \eqref{constraint-1} such that $\vec \Phi(y-x)=f\nabla \phi_0(y-x)$.
\end{proposition}
The proof of Proposition \ref{prop:Green-Vekua} relies on the Cauchy type integral formula \ref{eq:scalar-Vekua} and exploiting some properties of suitable fundamental solutions of $D-M^{\nabla f/f} C$. By direct computation, it was also proved in \cite{Delgado-Moreira} that if $\vec \Phi$ is a purely vector fundamental solution of $D-M^{\nabla f/f} C$, then it is equivalent to finding $\phi$ such that $\vec \Phi/f=\nabla \phi_0$ and the scalar function satisfies
\begin{align*}\label{eq:fundamental-conductivity}
\div_y\left(f^2(y) \nabla \phi_0(y-x)\right)=-f\delta(y-x), \qquad  y\in \Omega, \, x\in \mathbb{R}^n.
\end{align*}
We can construct the non-scalar part of solutions of the main Vekua equation given its scalar part as a solution of the conductivity equation \eqref{eq:conductivity-Vekua}, see \cite{DelPor2018} for the three-dimensional case and \cite{Delgado-Moreira} for the $n$-dimensional case.
By the above observations, we can set aside Vekua equations, rewrite the formula \eqref{eq:Green-Vekua-2} just in terms of $\phi$, and prove it using basically the Divergence theorem and the fact that the scalar function $\phi$ is a fundamental solution of the conductivity equation up to the multiplicative factor $f$.

In general, the construction of fundamental solutions of perturbed Dirac operators such as $D-(\nabla f/f)C$ and $D-M^{\nabla f/f}C$ has not been explored, even for the case when $\nabla f/f$ is a constant. Nevertheless, the standard Cauchy kernel defined in \eqref{eq:Cauchy-kernel} to be a purely vectorial fundamental solution of $D$ satisfies that $E/f$ is a fundamental solution of the Vekua operator $D-(\nabla f/f)C$ up to a multiplicative scalar function. Indeed,
\begin{equation}\label{eq:fundamental-Cauchy}
  \left(D-\dfrac{\nabla f}{f} C\right)\left(\dfrac{E}{f}\right)=\dfrac{DE}{f}-\dfrac{ \nabla f}{f^2}E+\dfrac{\nabla f}{f^2}E=\dfrac{\delta}{f}.
\end{equation}

\begin{proposition}\label{prop:integral-Cauchy}
Let $\Omega\subset \mathbb{R}^n$ be a bounded domain with Lipschitz boundary and $f\in W^{1,\infty}(\Omega)$ away from zero. If $w$ is a solution of the main Vekua equation \eqref{eq:main-Vekua} and $u_0=w_0/f=\text{Sc }(w)/f$ has boundary values and normal derivative well-defined, then the following formula holds in $\Omega$
\begin{align}\label{eq:green-formula-type-2}
&\int_{\partial \Omega} \left(-E(y-x)\cdot \eta(y)\, u_0|_{\partial\Omega}\,+\dfrac{1}{\sigma_n(n-2)}\frac{1}{|x-y|^{n-2}}\, \nabla u_0 |_{\partial\Omega}\cdot \eta\right) ds_{y}\nonumber\\
&+\int_{\Omega} \dfrac{2}{\sigma_n(n-2)}\frac{1}{|x-y|^{n-2}}\frac{\nabla f}{f}\cdot \nabla u_0 \, dy=\left\{
   \begin{array}{ll}
      u_0(x), \ & x\in \Omega,\\
      0,           & x\in \R^n\setminus \overline{\Omega}.
   \end{array}
\right.
\end{align}
\end{proposition}
\begin{proof}
The proof is a straightforward application of the Divergence theorem by noticing that $\dfrac{1}{\sigma_n(n-2)}\nabla_y\left(\frac{1}{|x-y|^{n-2}}\right)=E(y-x)$. Indeed,
\begin{align*}
&\text{ }\int_{\partial \Omega} \left(-E(y-x)\cdot \eta(y)\, u_0|_{\partial\Omega}\,+\dfrac{1}{\sigma_n(n-2)}\frac{1}{|x-y|^{n-2}}\, \nabla u_0 |_{\partial\Omega}\cdot \eta\right) ds_{y}\\
&=\int_{\Omega}\div\left(-E(y-x) u_0(y)+\dfrac{1}{\sigma_n(n-2)}\frac{1}{|x-y|^{n-2}}\, \nabla u_0\right)\, dy\\
&=\int_{\Omega} -\div E(y-x) u_0(y)\, dy+\dfrac{1}{\sigma_n(n-2)}\frac{1}{|x-y|^{n-2}}\, \Delta u_0 \, dy.
\end{align*}
The desired integral formula \eqref{eq:green-formula-type-2} comes from the facts that $\Sc D_yE(y-x)=-\div_y E(y-x)=\delta(y-x)$ and $\div(f^2\nabla u_0)=0$ in $\Omega$. 
\end{proof}
By the equivalence between the solutions of the main Vekua equation and the Clifford Beltrami equation provided in Proposition \ref{prop:Beltrami equivalence}, we have that the previous result can be enunciated in terms of the solutions of \eqref{eq:Clifford-Beltrami}. 
\section{Uniqueness result}\label{sec:uniqueness}
In this section, we revisit the reduction to the Schr\"{o}dinger equation. More precisely, we revisit that the uniqueness in the inverse conductivity problem can be deduced from the corresponding uniqueness result from the Schr\"{o}dinger equation, shown by Sylvester and Uhlmann in \cite{Sylvester-Uhlmann}, and carefully explained in \cite[Sec.\ 3.1]{Salo-2008}. Besides, we give a formula relating the difference of potentials to the difference of solutions to the associated Schr\"{o}dinger equations, see \eqref{eq:identity-potential}; in this formula, the interrelation between the fundamental solutions of the Vekua operators and the Schr\"{o}dinger operator plays a fundamental role. Later, applying the construction of complex geometrical optics solutions provided in \cite{Paivarinta2003}, we give an alternative proof of the uniqueness inverse problem for Lipschitz conductivities.

Since we work with Clifford-valued functions $w$ in our manuscript, we prefer to continue using the notation $w_0$ for scalar-valued functions to avoid confusion, except for instance for the conductivities $f, g$ where the context is clear.

If $q\in L^{\infty}(\Omega)$, we consider the Dirichlet problem for the Schr\"{o}dinger equation
\begin{equation}\label{eq:Schrodinger}
\left\{
	\begin{array}{rclll}
(-\Delta +q)w_0&=& 0 \quad \mbox{in}\quad \Omega,\\
w_0&=&\varphi_0\quad \mbox{on}\quad \partial\Omega,
\end{array}
\right. 
\end{equation}
where $\varphi_0\in H^{1/2}(\partial\Omega)$.
The problem \eqref{eq:Schrodinger} is well-posed if the following three conditions hold
\begin{enumerate}
    \item There exists a weak solution $w_0$ in $H^1(\Omega)$ for any boundary value $\varphi_0$ in $H^{1/2}(\partial\Omega)$.
    \item The solution $w_0$ is unique.
    \item The resolvent operator $\varphi_0\mapsto w_0$ is continuous from $H^{1/2}(\partial\Omega)$ to $H^1(\Omega)$.
\end{enumerate}
If the problem \eqref{eq:Schrodinger} is well-posed, we can define a Dirichlet-to-Neumann map as follows
\begin{align}\label{eq:D-N-Schrodinger}
\nonumber    \Lambda_{q}\colon H^{1/2}(\partial\Omega) &\to H^{-1/2}(\partial\Omega)\\
    \varphi_0 &\mapsto \nabla w_0|_{\partial\Omega}\cdot \eta.
\end{align}
The weak formulation of $\Lambda_q$ is
\begin{equation}\label{weak-D-N-Schrodinger}
(\Lambda_q \varphi_0,\psi_0)_{\partial\Omega}=\int_{\Omega} (\nabla w_0\cdot \nabla v_0+qw_0v_0)\, dy, \quad \varphi_0, \psi_0\in H^{1/2}(\partial\Omega),
\end{equation}
where $w_0$ solves \eqref{eq:Schrodinger} and $v_0$ is any function in $H^1(\Omega)$ such that $v_0|_{\partial\Omega}=\psi_0$. The problem \eqref{eq:Schrodinger} is not always well-posed. However, for potentials coming from conductivities, the Dirichlet problem is always well-posed.

Let us consider the Dirichlet problem for the conductivity equation
\begin{equation}\label{eq:conductivity-2}
	\left\{
\begin{array}{rclll}
\div\left(f^2\nabla(w_0/f)\right)&=& 0 \quad \mbox{in}\quad \Omega,\\
w_0&=&\varphi_0\quad \mbox {on}\quad \partial\Omega,
\end{array}
\right. 
\end{equation}
where $\varphi_0\in H^{1/2}(\partial\Omega)$ and $w_0\in H^1(\Omega)$ is the unique solution of \eqref{eq:conductivity-2}. We define the Dirichlet-to-Neumann map associated to \eqref{eq:conductivity-2} as follows
\begin{align}\label{eq:D-N-conductivity}
\nonumber    \Lambda_{f}\colon H^{1/2}(\partial\Omega) &\to H^{-1/2}(\partial\Omega)\\
    \varphi_0/f &\mapsto f^2\nabla (w_0/f)|_{\partial\Omega}\cdot \eta.
\end{align}
 Observe that the Dirichlet-to-Neumann map \eqref{eq:D-N-conductivity} is weakly defined by
\begin{align}\label{eq:weakly-D-N}
(\Lambda_f(\varphi_0/f),\psi_0)_{\partial\Omega}=\int_{\Omega} f^2 \nabla(w_0/f)\cdot \nabla v_0,
\end{align}
where $w_0$ solves the Dirichlet problem for the conductivity equation \eqref{eq:conductivity-2} and $v_0$ is any extension in $H^1(\Omega)$ such that $v_0|_{\partial\Omega}=\psi_0$.
The pairing in the boundary is defined as follows
\[
(\varphi,\psi)_{\partial\Omega}=\int_{\partial\Omega} \varphi(y) \psi(y)\, ds_y.
\]
In 1987, Sylvester and Uhlmann introduced the well-known \textit{complex geometrical optics solutions} (CGO) and proved the following uniqueness result for the Schr\"{o}dinger equation.
\begin{lemma}\label{lemma:uniqueness-potentials}
\cite{Sylvester-Uhlmann} Let $f, g$ be conductivities such that the corresponding potentials $q_f=\Delta f/f, q_g=\Delta g/g$ belong to $L^{\infty}(\Omega)$. If $\Lambda_{q_f}=\Lambda_{q_g}$, then $q_f=q_g$ in $\Omega$.
\end{lemma}
The following boundary determination result will be used throughout the manuscript; it was proved for smooth conductivities in \cite{Vogelius} and for $C^1$ conductivities in Lipschitz domains by \cite{Alessandrini}. 
\begin{lemma}\label{lemma:boundary}\cite{Alessandrini, Vogelius, Sylvester-Uhlmann-1988}
    Let $f, g\in W^{1,\infty}(\Omega)$ conductivities. If $\Lambda_f=\Lambda_g$, then $f|_{\partial\Omega}=g|_{\partial\Omega}$ and $\nabla f|_{\partial\Omega}\cdot \eta=\nabla g|_{\partial\Omega}\cdot \eta$.
\end{lemma}

\begin{lemma}\label{lemma-D-N}
Let $f\in W^{2,\infty}(\Omega)$ be conductivity away from zero and $q_f=\Delta f/f$, then the Dirichlet problem \eqref{eq:Schrodinger} is well-posed and the factorization is fulfilled
\[
-\operatorname{div} \left(f^2\nabla (w_0/f)\right)=f\left(-\Delta+q_f\right)w_0=0.
\]
Moreover, the following relation between the Dirichlet-to-Neumann maps $\Lambda_f$ and $\Lambda_{q_f}$ holds
\begin{equation}\label{eq:D-N-relation}
    \Lambda_{f}(\varphi_0/f)=f \Lambda_{q_f}(\varphi_0)-\nabla f|_{\partial\Omega}\cdot \eta \varphi_0.
\end{equation}
\end{lemma}
Now, we recall the proof of the uniqueness result provided by Sylvestar and Uhlmann: If $\Lambda_f=\Lambda_g$, then Lemmas \ref{lemma:boundary} and \ref{lemma-D-N} imply that $\Lambda_{q_f}=\Lambda_{q_g}$. By Lemma \ref{lemma:uniqueness-potentials}, we have that $\Delta f/f=\Delta g/g$ in $\Omega$. Since the Dirichlet problem for the Schr\"{o}dinger equation \eqref{eq:Schrodinger} is well-posed for potentials coming from conductivities and $f|_{\partial\Omega}=g|_{\partial\Omega}$, then $f\equiv g$ in $\Omega$.

Now, we revisit this result using the Green-Vekua integral formula developed in Section \ref{sec:integral}. Observe that for conductivities in $W^{2, \infty}(\Omega)$, we do not have enough regularity to write normal derivatives as we did in Proposition \ref{prop:Green-Vekua}. Thus, using the weak formulation of the Dirichlet-to-Neumann map $\Lambda_{f}$ \eqref{eq:weakly-D-N}, we have that the integral formula \eqref{eq:Green-Vekua-2} can be rewritten as follows
\begin{equation}\label{eq:Green-Vekua-3}
\int_{\partial \Omega} - \vec \Phi(y-x)\cdot \eta(y)\, w_0(y)\, ds_y+ \left(\Lambda_{f}(\varphi_0/f), \phi_0(\cdot-x)\right)_{\partial\Omega}=\left\{
   \begin{array}{ll}
      w_{0}(x), \ & x\in \Omega,\\
      0,           & x\in \R^n\setminus \overline{\Omega},
   \end{array}
\right.
\end{equation}
where $\vec \Phi$ is a purely vector fundamental solution of $D-M^{\nabla f/f} C$ satisfying \eqref{constraint-1} such that $\vec \Phi(y-x)=f\nabla \phi_0(y-x)$.
\begin{proposition}\label{prop:difference-potentials}
Let $\Omega\subset \mathbb{R}^n$ be a bounded domain with Lipschitz boundary. Let $f, g\in W^{2,\infty}(\Omega)$ conductivities away from zero. If $\Lambda_f=\Lambda_g$, then 
\begin{align}\label{eq:identity-potential}
\int_{\Omega} \left(\frac{\Delta g}{g}-\frac{\Delta f}{f}\right) w_{0,f}(y)\theta_g(y-x)\, dy&=\int_{\Omega} \left(\frac{\Delta g}{g}-\frac{\Delta f}{f}\right) w_{0,g}(y)\theta_f(y-x)\, dy \nonumber\\
&=\left\{
   \begin{array}{ll}
      w_{0,f}(x)-w_{0,g}(x), \ & x\in \Omega,\\
      0,           & x\in \R^n\setminus \overline{\Omega},
   \end{array}
\right.
\end{align}
where $w_{0,f}, w_{0,g}$ are solutions and $\theta_f, \theta_g$ are fundamental solutions of the Schr\"{o}dinger equation with potential $q_f=\Delta f/f$ and $q_g=\Delta g/g$, respectively. Moreover, the following statements are equivalent
\begin{enumerate}
    \item[(i)]  $w_{0,f}=w_{0,g}$ a.e. in $\Omega$.
     \item[(ii)] $\frac{\Delta f}{f}=\frac{\Delta g}{g}$ a.e. in $\Omega$.
\end{enumerate} 
\end{proposition}
\begin{proof}
Let $\varphi_0\in H^{1/2}(\partial\Omega)$, then there exists $w_{0,f}, w_{0,g}\in H^1(\Omega)$ such that $w_{0,f}/f$, $w_{0,g}/g$ are solutions of the conductivity equation \eqref{eq:conductivity-2} such that $w_{0,f}|_{\partial\Omega}=\varphi_0=w_{0,g}|_{\partial\Omega}$. Given $\theta_f$ and $\theta_g$ be scalar fundamental solutions of the Schr\"{o}dinger equation $-\Delta u+q_f u=0$ and $-\Delta u+q_g u=0$, respectively, with $q_f=\Delta f/f$ and $q_g=\Delta g/g$. 

By the factorization \eqref{eq:factorization-3}, we have that $\nabla \theta_f(y-x)-(\nabla f/f)\theta_f(y-x)$ and $\nabla \theta_g(y-x)-(\nabla g/g)\theta_g(y-x)$ are purely fundamental solutions of the operator $D-M^{\nabla f/f}C$ and $D-M^{\nabla g/g}C$, respectively.

Applying the Green-Vekua formula \eqref{eq:Green-Vekua-3}, we get that for $x\in \Omega$
\begin{align}\label{eq:computation-1}
\int_{\partial\Omega} -\left(\nabla \theta_f(y-x)-(\nabla f/f)\theta_f(y-x)\right)\cdot \eta \varphi_0 \, ds_y+\left(\Lambda_f(\varphi_0/f), f^{-1}\theta_f(\cdot-x)\right)_{\partial\Omega}&=w_{0,f}(x),\nonumber \\
\int_{\partial\Omega} -\left(\nabla \theta_g(y-x)-(\nabla g/g)\theta_g(y-x)\right)\cdot \eta \varphi_0 \, ds_y+\left(\Lambda_g(\varphi_0/g), g^{-1}\theta_g(\cdot-x)\right)_{\partial\Omega}&=w_{0,g}(x),
\end{align}
and both vanishes if $x\in \mathbb{R}^n\setminus \overline{\Omega}$. By the interrelation between the Dirichlet-to-Neumann maps associated to the conductivity equation \eqref{eq:D-N-conductivity} and Schr\"{o}dinger equation \eqref{eq:D-N-Schrodinger} with potentials $q_f=\Delta f/f$, $q_g=\Delta g/g$ (see Lemma \ref{lemma-D-N}), we get that \eqref{eq:computation-1} reduces to
\begin{align}\label{eq:computations-2}
\int_{\partial\Omega} -\nabla \theta_f(y-x)\cdot \eta \varphi_0 \, ds_y+\left(\Lambda_{q_f}(\varphi_0), \theta_f(\cdot-x)\right)_{\partial\Omega} =\left\{
   \begin{array}{ll}
      w_{0,f}(x), \ & x\in \Omega,\\
      0,           & x\in \R^n\setminus \overline{\Omega}.
   \end{array}
\right.\nonumber \\
\int_{\partial\Omega} -\nabla \theta_g(y-x)\cdot \eta \varphi_0 \, ds_y+ \left(\Lambda_{q_g}(\varphi_0), \theta_g(\cdot-x)\right)_{\partial\Omega} =\left\{
   \begin{array}{ll}
      w_{0,g}(x), \ & x\in \Omega,\\
      0,           & x\in \R^n\setminus \overline{\Omega}.
   \end{array}
\right.
\end{align}
Using that $f|_{\partial\Omega}=g|_{\partial\Omega},$ $\nabla f|_{\partial\Omega}\cdot \eta=\nabla g|_{\partial\Omega}\cdot \eta$, and the relation \eqref{eq:D-N-relation}, we easily get that  $\Lambda_{q_f}=\Lambda_{q_g}$. Subtracting both expressions in \eqref{eq:computations-2}, using the weak formulation \eqref{weak-D-N-Schrodinger}, and the fact that $\theta_f$, $\theta_g$ are fundamental solutions of the corresponding Schr\"{o}dinger equation, we get that for $x\in \Omega$
\begin{align}\label{eq:equivalences}
w_{0,f}(x)-w_{0,g}(x)&=\int_{\partial\Omega} -\left(\nabla \theta_f-\nabla \theta_g\right)\cdot \eta \varphi_0 \, ds_y+ \left(\Lambda_{q_f}(\varphi_0), \theta_f(\cdot-x)-\theta_g(\cdot-x)\right)_{\partial\Omega}\nonumber\\
&=\int_{\Omega} -\div(\left(\nabla \theta_f-\nabla \theta_g\right)w_{0,f})+ \left(\nabla \theta_f-\nabla \theta_g \right)\cdot \nabla w_{0,f}+\frac{\Delta f}{f}(\theta_f-\theta_g)w_{0,f} \, dy\nonumber\\
&=\int_{\Omega} \left(\frac{\Delta g}{g}-\frac{\Delta f}{f}\right) w_{0,f}(y)\theta_g(y-x)\, dy,
\end{align}
and it vanishes if $x\in \mathbb{R}^n\setminus \overline{\Omega}$. Similarly,
\begin{align*}
    \int_{\Omega} \left(\frac{\Delta g}{g}-\frac{\Delta f}{f}\right) w_{0,g}(y)\theta_f(y-x)\, dy=\left\{
   \begin{array}{ll}
      w_{0,f}(x)-w_{0,g}(x), \ & x\in \Omega,\\
      0,           & x\in \R^n\setminus \overline{\Omega}.
   \end{array}
\right.
\end{align*}
If $w_{0,f}=w_{0,g}$ a.e. in $\Omega$, then $0=\Delta(w_{0,f}-w_{0,g})=\left(\frac{\Delta g}{g}-\frac{\Delta f}{f}\right)w_{0,f}$, which easily implies that $\frac{\Delta g}{g}=\frac{\Delta f}{f}$ a.e. in $\Omega$. Now, if $\frac{\Delta g}{g}=\frac{\Delta f}{f}$ a.e. in $\Omega$, by \eqref{eq:equivalences}, we obtain that $w_{0,f}=w_{0,g}$ a.e. in $\Omega$, which ended the proof.
\end{proof}
Notice that the identity \eqref{eq:identity-potential} resembles the classical formula that relates the difference of potentials with the difference of boundary measurements $\Lambda_{q_f}-\Lambda_{q_g}$, namely
\[
\int_{\Omega}\left(\frac{\Delta f}{f}-\frac{\Delta g}{g}\right) w_{0,f}(y) w_{0,g}(y)\, dy=\left((\Lambda_{q_f}-\Lambda_{q_g})\varphi_{0,f}, \varphi_{0,g}\right)_{\partial\Omega},
\]
where $w_{0,f}, w_{0,g}$ are solutions of the Schr\"{o}dinger equation in $\Omega$ with potential $q_f=\Delta f/f$ and $q_g=\Delta g/g$; and boundary values $w_{0,f}|_{\partial\Omega}=\varphi_{0,f}$ and $w_{0,g}|_{\partial\Omega}=\varphi_{0,g}$, respectively.
The last identity is indeed the cornerstone of the proof of Lemma \ref{lemma:uniqueness-potentials} provided by Sylvester and Uhlmann in \cite{Sylvester-Uhlmann}. 
\subsection{Uniqueness result for Lipschitz conductivities}
In \cite[Th.\ 1.1]{Paivarinta2003}, P\"{a}iv\"{a}rinta, Panchenko and Uhlmann proved the existence of complex geometrical optics solutions for conductivities living in $W^{1,\infty}(\mathbb{R}^n),$ with $n\geq 3$. Let $-1<\delta<0$, let us consider the weighted $L^2$-space which is defined by 
\[
L_{\delta}^2(\mathbb{R}^n):=\left\{v\colon \int (1+|x|^2)^{\delta/2} |v(x)|^2\, dx<\infty\right\},
\]
and $H_{\delta}^2(\mathbb{R}^n)$ will denote the corresponding Sobolev space defined by interpolation.
\begin{theorem}\cite[Th.\ 1.1]{Paivarinta2003}\label{th:CGO}
If $f\in W^{1,\infty}(\mathbb{R}^n)$ and $f=1$ outside a large ball, then for $|\zeta|$ sufficiently large, there exists a unique solution of the conductivity equation 
\begin{equation}
    \div(f^2\nabla u)=0, \quad \text{ in }\mathbb{R}^n
\end{equation}
of the form 
\begin{equation}\label{eq:CGO-Lipschitz}
u=e^{x\cdot\zeta}\left(f^{-1}+\psi_f(x,\zeta)\right),
\end{equation}
with $\psi_f\in L_{\delta}^2(\mathbb{R}^n)$. Moreover, $\psi_f$ has the form
\begin{equation}
\psi_f(x,\zeta)=\left(\omega_0(x,\zeta)-f^{-1}\right)+\omega_1(x,\zeta),    
\end{equation}
where $\omega_0$ and $ \omega_1$ satisfy
\begin{equation}\label{eq:limits}
\lim_{|\zeta|\rightarrow \infty} \|\omega_0(x,\zeta)-f^{-1}\|_{H_{\delta}^1}=0 \text{ and } \lim_{|\zeta|\rightarrow \infty} \|\omega_1(x,\zeta)\|_{L_{\delta}^2}=0.
\end{equation}
\end{theorem}
 Observe that for Lipschitz conductivities, we do not have enough regularity to write normal derivatives as we did in Proposition \ref{prop:integral-Cauchy}. Thus, using the weak formulation of the Dirichlet-to-Neumann map \eqref{eq:weakly-D-N}, we have that the integral formula \eqref{eq:green-formula-type-2} can be rewritten as follows
\begin{align}\label{eq:rewrite}
&\int_{\partial \Omega} -E(y-x)\cdot \eta(y)\, (w_0/f)|_{\partial\Omega}\, ds_y \,+\, \left(\Lambda_{f}(\varphi_0),\dfrac{1}{\sigma_n(n-2)}\frac{f^{-2}}{|x-y|^{n-2}}\right)_{\partial\Omega}\nonumber\\
&+\int_{\Omega} \dfrac{2}{\sigma_n(n-2)}\frac{1}{|x-y|^{n-2}}\frac{\nabla f}{f}\cdot \nabla (w_0/f) \, dy=\left\{
   \begin{array}{ll}
      (w_0/f)(x), \ & x\in \Omega,\\
      0,           & x\in \R^n\setminus \overline{\Omega}.
   \end{array}
\right.
\end{align}
where $w_0/f$ solves the conductivity equation \eqref{eq:conductivity-Vekua} in $\Omega$ and $(w_0/f)|_{\partial\Omega}=\varphi_0$.

 Let $f, g\in W^{1,\infty}(\Omega)$ be conductivities with $\Lambda_f=\Lambda_g$. If we extend $f,g\in W^{1,\infty}(\mathbb{R}^n)$ with $f= g$ in $\R^n\setminus \overline{\Omega}$ and $f=g=1$ outside a large ball, then Sylvester and Uhlmann proved that the complex geometrical solutions as in \eqref{eq:CGO-Lipschitz}  coincides in $\R^n\setminus \overline{\Omega}$, see \cite{Sylvester-Uhlmann}. Now, we provide a similar result at the interior domain, our starting point is to analyze the integral representation formula \eqref{eq:rewrite}, and after applying a density argument we show that for any boundary value $\varphi_0\in H^{1/2}(\partial\Omega)$, the unique extensions in $H^1(\Omega)$ as solutions of the conductivity equation are identically in $\Omega$, whenever the boundary measurements provided by the Dirichlet-to-Neumann maps are equal. 

\begin{proposition}\label{prop:extensions-equal}
    Let $\Omega\subset \mathbb{R}^n$ be a bounded domain with Lipschitz boundary and let $f, g\in W^{1,\infty}(\Omega)$ conductivities away from zero. If $\Lambda_f=\Lambda_g$, then the Dirichlet-to-Neumann map uniquely determines the solutions of the conductivity equation in $\Omega$.
\end{proposition}
\begin{proof}
Let $\varphi_0\in H^{1/2}(\partial\Omega)$, then there exists $w_{0,f}, w_{0,g}\in H^1(\Omega)$ such that $w_{0,f}/f$, $w_{0,g}/g$ are solutions of the conductivity equation \eqref{eq:conductivity-Vekua} in $\Omega$ with $(w_{0,f}/f)|_{\partial\Omega}=\varphi_0=(w_{0,g}/g)|_{\partial\Omega}$.
By \eqref{eq:rewrite}, we have that for $x\in \Omega$
\begin{align}\label{eq:computation-6}
f^{-1}w_{0,f}(x)&=\int_{\partial\Omega} -E(y-x)\cdot \eta \varphi_0 \, ds_y \,+\, \left(\Lambda_{f}(\varphi_0),\dfrac{1}{\sigma_n(n-2)}\frac{f^{-2}}{|x-y|^{n-2}}\right)_{\partial\Omega} \\
&+\frac{2}{\sigma_n(n-2)}\int_{\Omega} \frac{1}{|x-y|^{n-2}} \frac{\nabla f}{f}\cdot \nabla(w_0/f)\, dy, \nonumber \\
g^{-1}w_{0,g}(x)&=\int_{\partial\Omega} -E(y-x)\cdot \eta \varphi_0 \, ds_y \,+\, \left(\Lambda_{g}(\varphi_0),\dfrac{1}{\sigma_n(n-2)}\frac{g^{-2}}{|x-y|^{n-2}}\right)_{\partial\Omega}\label{eq:computation-7}\\
&+\frac{2}{\sigma_n(n-2)}\int_{\Omega} \frac{1}{|x-y|^{n-2}}\frac{\nabla g}{g}\cdot \nabla(w_{0,g}/g)\, dy,\nonumber
\end{align}
and they are both identically zero if $x\in \mathbb{R}^n\setminus \overline{\Omega}$.
Let us denote by 
\begin{equation*}
\rho:=\frac{\nabla f}{f}\cdot \nabla(w_{0,f}/f)-\frac{\nabla g}{g}\cdot \nabla(w_{0,g}/g), 
\end{equation*}
with $\rho\in L^2(\Omega)$. Observe that the surface integrals in \eqref{eq:computation-6} and \eqref{eq:computation-7} are equal. By using that $\Lambda_f(\varphi_0)=\Lambda_g(\varphi_0)$ and $f|_{\partial\Omega}=g|_{\partial\Omega}$, then the pairings in the boundary satisfy
\begin{equation*}
    \left(\Lambda_{f}(\varphi_0),\dfrac{1}{\sigma_n(n-2)}\frac{f^{-2}}{|x-y|^{n-2}}\right)_{\partial\Omega}=\left(\Lambda_{g}(\varphi_0),\dfrac{1}{\sigma_n(n-2)}\frac{g^{-2}}{|x-y|^{n-2}}\right)_{\partial\Omega}.
\end{equation*}
Therefore,
\begin{align}\label{eq:Newton-potential}
\frac{2}{\sigma_n(n-2)}\int_{\Omega} \frac{\rho(y)}{|x-y|^{n-2}}\, dy=\left\{
   \begin{array}{ll}
      f^{-1}w_{0,f}(x)-g^{-1}w_{0,g}(x), \ & x\in \Omega,\\
      0,           & x\in \R^n\setminus \overline{\Omega}.
   \end{array}
\right.
\end{align}
Notice that the left-hand side of \eqref{eq:Newton-potential} corresponds to the Newton potential acting on the function $\rho\in L^2(\Omega)$. By potential theory, we get that $f^{-1}w_{0,f}-g^{-1}w_{0,g}\in H^2(\Omega)$.
Consequently, $\Delta (f^{-1}w_{0,f}-g^{-1}w_{0,g})$ is well-defined 
and applying the Laplacian operator in both sides of \eqref{eq:Newton-potential}, we obtain that
\begin{align}\label{eq:uniqueness-1}
\Delta(f^{-1}w_{0,f}-g^{-1}w_{0,g})(x)=-2\rho(x), \qquad x\in \Omega.
\end{align}
Let $k\in \mathbb{C}^n$ be fixed and $u\in C^2(\Omega)\cap C^1(\overline{\Omega})$. By the Green's identity, we get
\begin{align*}
    0&=\int_{\partial\Omega} u|_{\partial\Omega}\nabla (f^{-1}w_{0,f}-g^{-1}w_{0,g})|_{\partial\Omega}\cdot \eta-(f^{-1}w_{0,f}-g^{-1}w_{0,g})|_{\partial\Omega}\nabla u|_{\partial\Omega}\cdot \eta\, ds_y\\
    &=\int_{\Omega}\Delta u\,(f^{-1}w_{0,f}-g^{-1}w_{0,g})-u\, \Delta(f^{-1}w_{0,f}-g^{-1}w_{0,g})\,  dy\\
    &=\int_{\Omega}\Delta u\,(f^{-1}w_{0,f}-g^{-1}w_{0,g})+ 2u \rho\, dy
\end{align*}
Taking $u(y)=e^{iy\cdot k}$, then 
\begin{align*}
0=\int_{\Omega} \left(-|k|^2(f^{-1}w_{0,f}-g^{-1}w_{0,g})+2\rho\right)e^{iy \cdot k}\, dy.
\end{align*}
Since it holds for every $k\in \mathbb{C}^n$, then by Fourier inversion we get
\begin{equation}\label{eq:rho}
2\rho=|k|^2(f^{-1}w_{0,f}(x)-g^{-1}w_{0,g}), \qquad \text{ in } \Omega.     
\end{equation}
Substituting \eqref{eq:rho} into \eqref{eq:uniqueness-1}, we get that $\Delta(f^{-1}w_{0,f}-g^{-1}w_{0,g})=-|k|^2(f^{-1}w_{0,f}(x)-g^{-1}w_{0,g})$. Choosing $k_1,k_2\in \mathbb{C}^n$ such that $|k_1|^2\ne |k_2|^2$, then $f^{-1}w_{0,f}-g^{-1}w_{0,g}$ simultaneously satisfies $\Delta(f^{-1}w_{0,f}-g^{-1}w_{0,g})=-|k_i|^2(f^{-1}w_{0,f}(x)-g^{-1}w_{0,g})$ in $\Omega$, $i=1,2$. Subtracting them, we obtain that $(|k_1|^2-|k_2|^2)(f^{-1}w_{0,f}-g^{-1}w_{0,g})=0$ in $\Omega$. Therefore, 
\begin{equation*}
\frac{w_{0,f}}{f}=\frac{w_{0,g}}{g}, \qquad \text{ in }\Omega.
\end{equation*}
In conclusion, the solutions of the conductivity equation are uniquely determined by the boundary measurements provided by the Dirichlet-to-Neumann maps, as we desired.
\end{proof}
Observe that by \eqref{eq:rho} or \eqref{eq:uniqueness-1}, an immediate consequence is that $\rho\equiv 0$ in $\Omega$. That is,
\begin{align}\label{eq:uniqueness-2}
\frac{\nabla f}{f}\cdot \nabla(w_{0,f}/f)=\frac{\nabla g}{g}\cdot \nabla(w_{0,g}/g), \qquad \text{ in }\Omega.
\end{align}

Since our conductivities are not sufficiently regular to define the potentials $q_f=\Delta f/f$ pointwise, we can adopt the distributional setting proposed by Brown \cite{Brown-1996}. Indeed, for $\psi\in C_c^{\infty}(\Omega)$, we may define the product $q_f w_{0,f}$ in the distributional sense as follows
\begin{align}\label{eq:weak-product}
    m_{q_f}(w_{0,f})(\psi):=\Biggl\langle \frac{\Delta f}{f} w_{0,f}, \psi \Biggr\rangle&=-\int_{\Omega} \nabla f\cdot \nabla\left(\frac{w_{0,f}\psi}{f}\right)\, dy.
\end{align}
For conductivities in $W^{2,\infty}(\Omega)$, we previously proved some equivalences in Proposition \ref{prop:difference-potentials}; now we present a stronger result, which directly implies the uniqueness result for this kind of conductivities.
\begin{proposition}\label{prop:equivalences}
    Let $\Omega\subset \mathbb{R}^n$ be a bounded domain with Lipschitz boundary and let $f, g\in W^{1,\infty}(\Omega)$ be proper conductivities away from zero. If $\Lambda_f=\Lambda_g$, then the following statements are equivalent
\begin{enumerate}
    \item[(i)]  $w_{0,f}=w_{0,g}$ a.e. in $\Omega$.
     \item[(ii)] $f=g$ a.e. in $\Omega$.
     \item[(iii)] $m_{q_f}(w_{0,f})=m_{q_g}(w_{0,g})$ a.e. in $\Omega$.
\end{enumerate} 
\end{proposition}
\begin{proof}
Let $\varphi_0\in H^{1/2}(\partial\Omega)$, then there exists $w_{0,f}, w_{0,g}\in H^1(\Omega)$ such that $w_{0,f}/f$, $w_{0,g}/g$ are solutions of the conductivity equation \eqref{eq:conductivity-Vekua} in $\Omega$ with $(w_{0,f}/f)|_{\partial\Omega}=\varphi_0=(w_{0,g}/g)|_{\partial\Omega}$. Then $w_{0,f}$, $w_{0,g}$ satisfy in the distributional sense
\begin{equation*}
    \Delta w_{0,f}-m_{q_f}(w_{0,f})=0, \qquad \Delta w_{0,g}-m_{q_g}(w_{0,g})=0,
\end{equation*}
where $m_q$ was defined in \eqref{eq:weak-product}.

\textbf{(i) implies (ii):} If $w_{0,f}=w_{0,g}$ a.e. in $\Omega$, then by Proposition \ref{prop:extensions-equal} $w_{0,f}/f=w_{0,g}/g$, which implies that $f=g$ a.e. in $\Omega$.

\textbf{(ii) implies (iii):} If $f=g$ a.e. in $\Omega$, by definition $m_{q_f}(\cdot)=m_{q_g}(\cdot)$. Again by Proposition \ref{prop:extensions-equal}, we get that $w_{0,f}=w_{0,g}$ a.e. in $\Omega$. So, $m_{q_f}(w_{0,f})=m_{q_g}(w_{0,g})$ a.e. in $\Omega$.

\textbf{(iii) implies (i):} If $m_{q_f}(w_{0,f})=m_{q_g}(w_{0,g})$, then $w_{0,f}-w_{0,g}$ is a harmonic function with vanishing trace, so $w_{0,f}=w_{0,g}$ in $\Omega$.
This completes the proof.
\end{proof}
Now, we are ready to give an alternative proof of the uniqueness result for the Calderón problem for conductivities living in $W^{1,\infty}(\Omega)$ away from zero.

\begin{proof}[\textbf{Proof of Theorem \ref{th:uniqueness}}]
Let $\varphi_0\in H^{1/2}(\partial\Omega)$, then there exists $w_{0,f}, w_{0,g}\in H^1(\Omega)$ such that $w_{0,f}/f$, $w_{0,g}/g$ are solutions of the conductivity equation \eqref{eq:conductivity-Vekua} in $\Omega$ with $(w_{0,f}/f)|_{\partial\Omega}=\varphi_0=(w_{0,g}/g)|_{\partial\Omega}$. 

If we extend $f,g\in W^{1,\infty}(\mathbb{R}^n)$ with $f\equiv g$ in $\R^n\setminus \overline{\Omega}$ and $f=g=1$ outside a big ball, then following the construction of complex geometrical solutions provided in Theorem \ref{th:CGO}, then for $\zeta\in \mathbb{C}^n$ with $|\zeta|$ sufficiently large
\begin{align}\label{eq:CGO}
\frac{w_{0,f}}{f}(x)=e^{x\cdot\zeta}\left(f^{-1}+\psi_f(x,\zeta)\right), \quad \frac{w_{0,g}}{g}(x)=e^{x\cdot\zeta}\left(g^{-1}+\psi_g(x,\zeta)\right)
\end{align}
where $\psi_{f}, \psi_{g}\in L^2_{\delta}(\mathbb{R}^n)$.
Moreover, 
\begin{align*}
\psi_{f}(x,\zeta)=\left(\omega_{0,f}(x,\zeta)-f^{-1}\right)+\omega_{1,f}(x,\zeta), \quad \psi_{g}(x,\zeta)=\left(\omega_{0,g}(x,\zeta)-g^{-1}\right)+\omega_{1,g}(x,\zeta),
\end{align*}
where
\begin{align*}
\lim_{|\zeta|\rightarrow \infty} \|\omega_{0,f}(x,\zeta)-f^{-1}\|_{H_{\delta}^1}=0, &\quad \lim_{|\zeta|\rightarrow \infty} \|\omega_{1,f}(x,\zeta)\|_{L_{\delta}^1}=0,\\
\lim_{|\zeta|\rightarrow \infty} \|\omega_{0,g}(x,\zeta)-f^{-1}\|_{H_{\delta}^1}=0, &\quad \lim_{|\zeta|\rightarrow \infty} \|\omega_{1,g}(x,\zeta)\|_{L_{\delta}^1}=0. 
\end{align*}
By Proposition \ref{prop:extensions-equal}, we know that $w_{0,f}/f=w_{0,g}/g$ in $\Omega$. Consequently,
\begin{equation*}
    e^{x\cdot\zeta}\left(f^{-1}+\psi_f(x,\zeta)\right)=e^{x\cdot\zeta}\left(g^{-1}+\psi_g(x,\zeta)\right)
\end{equation*}
By taking the limit as $|\zeta|\rightarrow \infty$, we get that $f\equiv g$ in $\Omega$, ending the proof of Theorem \ref{th:uniqueness}.
\end{proof}




\end{document}